\documentclass{amsart}

\usepackage[mathscr]{eucal}
\usepackage{amssymb}
\usepackage[usenames,dvipsnames]{color}
\usepackage{amsthm}
\usepackage{bbold}
\usepackage{enumerate}
\usepackage{array}
\usepackage{amsmath}

\usepackage[colorlinks=true,linkcolor={Brown},citecolor={Brown},urlcolor={Brown},hypertex]{hyperref}

\numberwithin{equation}{section}
\usepackage[all]{xy}
\SelectTips{cm}{}

\newdir{ >}{{}*!/-10pt/\dir{>}}

\newtheorem*{Thm*}{Theorem}
\newtheorem{Cor}[equation]{Corollary}
\newtheorem{Lem}[equation]{Lemma}
\newtheorem{Prop}[equation]{Proposition}
\newtheorem{Thm}[equation]{Theorem}

\theoremstyle{remark}
\newtheorem{Def}[equation]{Definition}

\newtheorem{Hyp}[equation]{Hypotheses}

\newtheorem{Rem}[equation]{Remark}
\newtheorem{Que}[equation]{Question}

\newcommand{\nc}{\newcommand}
\nc{\dmo}{\DeclareMathOperator}

\dmo{\coker}{coker}
\dmo{\colim}{colim}
\dmo{\cone}{cone}
\dmo{\End}{End}
\dmo{\Hom}{Hom}
\dmo{\Id}{Id}
\dmo{\Ind}{Ind}
\dmo{\Ker}{Ker}
\dmo{\opname}{op}
\dmo{\supp}{supp}
\dmo{\Spc}{Spc}
\dmo{\Spec}{Spec}

\nc{\bbN}{\mathbb{N}}
\nc{\bbQ}{\mathbb{Q}}
\nc{\bbR}{\mathbb{R}}
\nc{\bbZ}{\mathbb{Z}}
\nc{\cat}[1]{\mathscr{#1}}
\nc{\eg}{{\sl e.g.}}
\nc{\Endcat}[1]{\End_{\cat #1}}
\nc{\Homcat}[1]{\Hom_{\cat #1}}
\nc{\ie}{{\sl i.e.}\ }
\nc{\inv}{^{-1}}
\nc{\isoto}{\overset{\sim}{\,\to\,}}
\nc{\isotoo}{\overset{\sim}{\,\too\,}}
\nc{\op}{^{\opname}}
\nc{\oto}[1]{\overset{#1}\to}
\nc{\otoo}[1]{\overset{#1}{\,\too\,}}
\nc{\Paul}[1]{{\color{Violet}#1}}
\nc{\qquadtext}[1]{\qquad\textrm{#1}\qquad}
\nc{\quadtext}[1]{\quad\textrm{#1}\quad}
\nc{\SET}[2]{\big\{\,#1\,\big|\,#2\,\big\}}
\nc{\too}{\mathop{\longrightarrow}\limits}
\nc{\unit}{\mathbb{1}}

\dmo{\DM}{DM}
\dmo{\im}{im}
\dmo{\SH}{SH}

\nc{\ababs}{{\sl ab absurdo}}
\nc{\adhpt}[1]{\overline{\{#1\}}}
\nc{\adj}{\dashv}
\nc{\bbA}{\mathbb{A}}
\nc{\DMc}{\DM_{\text{gm}}}
\nc{\eps}{\epsilon}
\nc{\ideal}[1]{\langle #1\rangle}
\nc{\oursetminus}{\!\smallsetminus\!}
\nc{\potimes}[1]{^{\otimes #1}}
\nc{\sbull}{{\scriptscriptstyle\bullet}}
\nc{\SpcK}{\Spc(\cat K)}
\nc{\SpcL}{\Spc(\cat L)}
\nc{\then}{\Rightarrow}

\nc{\tnil}{$\otimes$-nilpotent}
\nc{\tnilce}{$\otimes$-nilpotence}
\nc{\cM}{\cat{M}}
\nc{\cP}{\cat{P}}
\nc{\cQ}{\cat{Q}}

\begin{document}


\title[On surjectivity on tensor-triangular spectra]{On the surjectivity of the map of spectra\\associated to a tensor-triangulated functor}
\author{Paul Balmer}
\date{2018 March 22}

\address{Paul Balmer, Mathematics Department, UCLA, Los Angeles, CA 90095-1555, USA}
\email{balmer@math.ucla.edu}
\urladdr{http://www.math.ucla.edu/$\sim$balmer}

\begin{abstract}
We prove a few results about the map $\Spc(F)$ induced on tensor-triangular spectra by a tensor-triangulated functor~$F$. First, $F$ is conservative if and only if $\Spc(F)$ is surjective on closed points. Second, if $F$ detects tensor-nilpotence of morphisms then $\Spc(F)$ is surjective on the whole spectrum. In fact, surjectivity of~$\Spc(F)$ is equivalent to $F$ detecting the nilpotence of some class of morphisms, namely those morphisms which are nilpotent on their cone.
\end{abstract}

\subjclass[2010]{18E30; 14F42, 19K35, 55U35}
\keywords{Tensor-triangular, spectra, nilpotence, conservative, surjectivity}

\thanks{Research supported by NSF grant~DMS-1600032.}

\maketitle


\section{Introduction}


\begin{Hyp}
\label{hyp}%
Throughout the paper, $F\colon\cat{K}\to \cat{L}$ is a tensor-triangulated functor between essentially small tensor-triangulated categories~$\cat{K}$ and~$\cat{L}$. Assume that $\cat{K}$ is \emph{rigid}, \ie every object has a dual (Remark~\ref{rem:rigid}).
\end{Hyp}

Consider the induced map on spectra
\[
\varphi=\Spc(F)\colon \SpcL\to \SpcK
\]
in the sense of tensor-triangular geometry~\cite{Balmer05a,BalmerICM,Stevenson16pp}. Our first result is a characterization of conservativity of~$F$.

\begin{Thm}
\label{thm:surj-closed}%
Under Hypotheses~\ref{hyp}, the following properties are equivalent:
\begin{enumerate}[\rm(a)]
\item
The functor $F\colon\cat{K}\to \cat{L}$ is conservative, \ie it detects isomorphisms.
\smallbreak
\item
The induced map~$\varphi\colon\SpcL\to \SpcK$ is surjective on closed points, \ie for every closed point $\cP$ in~$\SpcK$, there exists $\cQ$ in~$\SpcL$ such that $\varphi(\cQ)=\cP$.
\end{enumerate}
\end{Thm}

We can remove the assumption that $\cat{K}$ is rigid, at the cost of replacing~(a) by:
\begin{enumerate}[\rm(a')]
\item
\label{it:a'}%
$F$ detects \tnilce\ of objects, \ie $F(x)=0 \then x\potimes{n}=0$ for some $n\ge 1$.
\end{enumerate}

\medbreak

Our main results are dedicated to surjectivity of~$\varphi$ on the whole of~$\SpcK$.

\begin{Thm}
\label{thm:surj-nil}%
Under Hypotheses~\ref{hyp}, suppose that the functor $F\colon\cat{K}\to \cat{L}$ detects \tnilce\ of morphisms, \ie every $f\colon x\to y$ in~$\cat{K}$ such that $F(f)=0$ satisfies $f\potimes{n}=0$ for some~$n\ge1$. Then the induced map~$\varphi\colon\SpcL\to \SpcK$ is surjective.
\end{Thm}

This result is clearly a corollary of (b)$\then$(a) in the following more technical result:
\begin{Thm}
\label{thm:main}%
Under Hypotheses~\ref{hyp}, the following properties are equivalent:
\begin{enumerate}[\rm(a)]
\item
The morphism $\varphi\colon\SpcL\to \SpcK$ is surjective.
\smallbreak
\item
\label{it:nil-cone}%
The functor $F\colon \cat{K}\to \cat{L}$ detects \tnilce\ of morphisms which are already \tnil\ on their cone, \ie every $f\colon x\to y$ in~$\cat{K}$ such that $F(f)=0$ and such that $f\potimes{m}\otimes\cone(f)=0$ for some $m\ge 1$ satisfies $f\potimes{n}=0$ for some~$n\ge1$.
\end{enumerate}
\end{Thm}

At this point, the Devinatz-Hopkins-Smith~\cite{DevinatzHopkinsSmith88} Nilpotence Theorem might come to some readers' mind. This celebrated result asserts that a morphism between finite objects in the topological stable homotopy category~$\SH$ must be $\otimes$-nilpotent if it vanishes on complex cobordism. Hopkins and Smith used the Nilpotence Theorem in the subsequent work~\cite{HopkinsSmith98} to prove the Chromatic Tower Theorem. A reformulation of the latter, in terms of~$\Spc(\SH^c)$, can be found in~\cite[\S\,9]{Balmer10b}. From the Nilpotence Theorem it follows that every prime of~$\SH^c$ is the kernel of some Morava $K$-theory. This implication is analogous to the surjectivity of Theorem~\ref{thm:surj-nil} in the special case of~$\SH$.

Let us stress however that the scope of Theorems~\ref{thm:surj-closed} and~\ref{thm:surj-nil} is broader than the topological example. In fact, $\SH$ plays among general tensor-triangulated categories the same role that $\bbZ$ plays among general commutative rings. Commutative algebra is not only the study of $\bbZ$, and tt-geometry is not only the study of~$\SH$. For the reader who never heard of tensor-triangulated categories and yet had the fortitude to read thus far, let us recall that tt-categories also appear in algebraic geometry (\eg\ derived categories of schemes), in representation theory (\eg\ derived and stable categories of finite groups), in noncommutative topology (\eg\ $KK$-categories of $C^*$-algebras), in motivic theory (\eg\ stable $\bbA^1$-homotopy and derived categories of motives), and in equivariant analogues (\eg\ equivariant stable homotopy theory). A good introduction can be found in~\cite[\S\,1.2]{HoveyPalmieriStrickland97}. Tensor-triangular geometry is an umbrella theory for all those examples. In particular, computing $\SpcK$ is \emph{the} fundamental problem for every tt-category~$\cat{K}$ out there; see~\cite[Thm.\,4.10]{Balmer05a}.

\medbreak

After this motivational digression, let us return to the development of our results. It is interesting to know whether the converse of Theorem~\ref{thm:surj-nil} holds true in glorious generality: Does surjectivity of~$\Spc(F)$ alone guarantee that $F$ detects \tnilce\ of morphisms? By Theorem~\ref{thm:main}, this problem can be reduced as follows.
\begin{Que}
\label{que}%
Under Hypotheses~\ref{hyp}, if $\varphi\colon\SpcL\to \SpcK$ is surjective and if $f\colon x\to y$ satisfies $F(f)=0$, is $f$ necessarily \tnil\ on its cone?
\end{Que}

We do not know any counter-example. In fact, we can give a positive answer under the assumption that $F\colon \cat{K}\to \cat{L}$ admits a right adjoint. Since $\cat{K}$ and $\cat{L}$ are essentially small (typically the `compact' objects of some big ambient category), existence of such a right adjoint is rather restrictive. In the context  of~\cite{BalmerDellAmbrogioSanders16}, it would be equivalent to having `Grothendieck-Neeman' duality. To give an example, this right adjoint exists in the case of a finite separable extension, see~\cite{Balmer16b}. The following are generalizations of some of the results in~\cite{Balmer16a}.

\begin{Thm}
\label{thm:FU}%
Under Hypotheses~\ref{hyp}, suppose that $F\colon \cat{K}\to \cat{L}$ admits a right adjoint~$U\colon\cat{L}\to \cat{K}$. Then the map $\varphi\colon\SpcL\to \SpcK$ is surjective if and only if the functor $F\colon\cat{K}\to \cat{L}$ detects \tnilce\ of morphisms.
\end{Thm}

Again, this is a special case of a sharper, slightly more technical result.

\begin{Thm}
\label{thm:FU+}%
Under Hypotheses~\ref{hyp}, suppose that $F\colon \cat{K}\to \cat{L}$ admits a right adjoint~$U\colon\cat{L}\to \cat{K}$ and consider the image~$U(\unit)\in\cat{K}$ of the $\otimes$-unit. Then the image of the map $\varphi\colon\SpcL\to \SpcK$ is exactly the support of the object~$U(\unit)$:
\[
\im(\Spc(F))=\supp(U(\unit))\,.
\]
\end{Thm}

An example of the latter, not covered by the separable extensions of~\cite{Balmer16a}, can be obtained by `modding out' coefficients in motivic categories, see~\cite[Chap.\,5]{VoevodskySuslinFriedlander00}. For instance, if $\cat{K}=\DMc(X;\bbZ)\otoo{F}\DMc(X;\bbZ/p)=\cat{L}$ then we have $\im(\Spc(F))=\supp(\bbZ/p)$. From these techniques, one can easily reduce the computation of the (yet unknown) spectrum of the integral derived category of geometric motives~$\DMc(X,\bbZ)$ to the case of field coefficients:
\[
\Spc(\DMc(X;\bbZ))=\im(\Spc(\DMc(X;\bbQ)))\sqcup \ \bigsqcup_{p}\ \im(\Spc(\DMc(X;\bbZ/p)))\,.
\]
These considerations will be pursued elsewhere.

In the presence of a `big' ambient category, our condition of detecting $\otimes$-nilpotence could also be related to conservativity, as discussed in~\cite[Thm.\,4.19]{MathewNaumannNoel17}.

\medbreak

Let us now state a direct consequence of Theorem~\ref{thm:surj-nil}, that was apparently never noticed despite its importance and simplicity. It is the case where $F$ is faithful.
\begin{Cor}
\label{cor:surj-ff}%
Suppose that $\cat{K}\subset\cat{L}$ is a rigid tensor-triangulated subcategory. Then every prime $\cP\in\SpcK$ is the intersection of a prime~$\cQ\in\SpcL$ with~$\cat{K}$.
\end{Cor}

A special sub-case of interest is that of `cellular' subcategories, \ie those $\cat{K}\subseteq\cat{L}$ generated by a collection of `nice' objects of~$\cat{L}$, typically $\otimes$-invertible ones (spheres). Such cellular subcategories~$\cat{K}$ are commonly studied when the ambient~$\cat{L}$ appears out-of-reach of known methods. For instance, Dell'Ambrogio~\cite{DellAmbrogio10} used this approach for equivariant $KK$-theory, and later with Tabuada~\cite{DellAmbrogioTabuada12} for non-commutative motives. Peter~\cite{Peter13} discusses the case of mixed Tate motives. Similarly, Heller-Ormsby~\cite{HellerOrmsby16pp} consider cellular subcategories in their recent study of tt-geometry in stable motivic homotopy theory. In all cases, Corollary~\ref{cor:surj-ff} says that whatever can be detected via these cellular subcategories~$\cat{K}$ is actually relevant information about the bigger and more mysterious ambient category~$\cat{L}$. In particular, surjectivity of the comparison homomorphisms introduced in~\cite{Balmer10b} can be tested on the cellular subcategory:
\begin{Cor}
\label{cor:surj-rho}%
Let $u\in\cat{L}$ be a $\otimes$-invertible object and $\cat{K}$ the full thick triangulated subcategory of~$\cat{L}$ generated by~$\SET{u\potimes{n}}{n\in\bbZ}$, which is supposed rigid\,\textrm{\rm(\footnote{\,This is automatic if $\cat{L}$ lives in a `big' ambient category with internal hom, where rigid objects are closed under triangles. See~\cite[Thm.\,A.2.5\,(a)]{HoveyPalmieriStrickland97}.})}. Note that the graded rings $R^\sbull_{\cat{K},u}$ and $R^\sbull_{\cat{L},u}$ associated to~$u$ are the same in~$\cat{K}$ and in~$\cat{L}$:
\[
R^\sbull_{\cat{K},u}\ \overset{\textrm{def}}=\ \Homcat{K}(\unit, u\potimes{\sbull})=\Homcat{L}(\unit, u\potimes{\sbull})\ \overset{\textrm{def}}=\ R^\sbull_{\cat{L},u}\,.
\]
If the comparison map $\rho^\sbull_{\cat{K},u}$ for~$\cat{K}$ (recalled below) is surjective for the `cellular' subcategory~$\cat{K}$ then the comparison map $\rho^\sbull_{\cat{L},u}$ for the ambient~$\cat{L}$ is also surjective:
\[
\xymatrix@R=2em{
\SpcL \ \ar@{->>}[r]^-{\textrm{Cor.\,\ref{cor:surj-ff}}} \ar[d]_-{\rho^\sbull_{\cat{L},u}} \ar@{}[rd]|-{\circlearrowright}
& \ \SpcK \ar[d]^-{\rho^\sbull_{\cat{K},u}}
& \cP \ar@{|->}[d]_-{\rho^\sbull_{\cat{K},u}}  \ar@{}[l]|(.7){\ni}
\\
\Spec^\sbull(R^\sbull_{\cat{L},u}) \ \ar@{=}[r]
& \ \Spec^\sbull(R^\sbull_{\cat{K},u})
& \rho^\sbull_{\cat{K},u}(\cP)\ \overset{\textrm{def}}=\ \SET{f\in R^\sbull_{\cat{K},u}}{\cone(f)\notin\cP} \ar@{}[l]|(.7){\ni} \,.\!\!
}
\]
\end{Cor}

For an introduction to these comparison maps and their importance, the reader is invited to consult the above references~\cite{Balmer10b,DellAmbrogio10,DellAmbrogioTabuada12,HellerOrmsby16pp} or~\cite{Sanders13}.

\medbreak
\noindent
\textbf{Acknowledgments}: I am thankful to Beren Sanders for observing in a previous version of this article that my proof of surjectivity of~$\varphi$ reduced to $F$ detecting nilpotence of morphisms of the form $\eta_x\otimes y\colon y\to x^\vee\otimes x\otimes y$. Beren's idea led me to the `morphisms which are nilpotent on their cone' and to Theorem~\ref{thm:main}. I also thank Ivo Dell'Ambrogio, Martin Gallauer, Jeremiah Heller and Kyle Ormsby for their comments.

\goodbreak
\section{The proofs}
\label{se:proofs}%
\medbreak

The tensor $\otimes\colon\cat{K}\times\cat{K}\too \cat{K}$ is exact in each variable and~$\unit$ stands for the $\otimes$-unit in~$\cat{K}$. Recall that a \emph{tt-ideal $\cat{J}\subseteq \cat{K}$} is a triangulated, thick, $\otimes$-ideal subcategory, \ie it is non-empty, is closed under taking cones, direct summands and under tensoring by any object of~$\cat{K}$. For $\mathcal{E}\subseteq \cat{K}$, we denote by~$\ideal{\mathcal{E}}\subseteq\cat{K}$ the tt-ideal it generates.

A proper tt-ideal $\cP\subsetneq\cat{K}$ is \emph{prime} if $x\otimes y\in\cP$ implies $x\in\cP$ or $y\in\cP$. The \emph{spectrum} $\SpcK=\SET{\cP\subset\cat{K}}{\cP\textrm{ is prime}}$ has a topology whose basis of open is given by the subsets $U(x)=\SET{\cP\in\SpcK}{x\in \cP}$, for every~$x\in\cat{K}$. The closed complement $\supp(x)=\SET{\cP\in\SpcK}{x\notin\cP}$ is called the \emph{support} of the object~$x$. A tensor-triangulated functor $F\colon\cat{K}\to\cat{L}$ induces a continuous map $\varphi=\Spc(F)\colon\SpcL\to \SpcK$ given explicitly by~$\varphi(\cQ)=F\inv(\cQ)$, for every prime~$\cQ\subset\cat{L}$.

\begin{Rem}
\label{rem:rigid}%
Our assumption that the tensor category~$\cat{K}$ is \emph{rigid}, means that there exists an exact functor called the \emph{dual}
\[
(-)^\vee\colon \cat{K}\op \too \cat{K}
\]
that provides an adjoint to tensoring with any object $x\in \cat{K}$ as follows:
\begin{equation}
\label{eq:adj-dual}
\vcenter{\xymatrix@R=2em{
\cat{K} \ar@<-.5em>[d]_-{x\otimes-} \ar@{}[d]|-{\adj}
\\
\cat{K} \ar@<-.5em>[u]_-{x^\vee\otimes-}
}}
\end{equation}
Some authors call such objects~$x$ \emph{strongly dualizable}, \eg~\cite{HoveyPalmieriStrickland97}. The adjunction~\eqref{eq:adj-dual} comes with units (coevaluation) and counits (evaluation)
\begin{equation}
\label{eq:(co)units}%
\eta_x\colon\unit\to x^\vee\otimes x
\qquadtext{and}
\eps_x\colon x\otimes x^\vee\to \unit
\end{equation}
which satisfy the relation
\begin{equation}
\label{eq:(co)unit-relation}%
(\eps_x\otimes x)\circ (x\otimes \eta_x)=1_x\,.
\end{equation}
It follows from~\eqref{eq:(co)unit-relation} that $x$ is a direct summand of $x\otimes x^\vee \otimes x\cong x\potimes{2}\otimes x^\vee$.

It is a general fact that any tensor functor $F\colon\cat{K}\to \cat{L}$ preserves rigidity, since we can use $F(x^\vee)$ as $F(x)^\vee$ with $F(\eta_x)$ and $F(\eps_x)$ as units and counits. See for instance~\cite[Prop.\,3.1]{FauskHuMay03}. In particular, although we do not assume $\cat{L}$ rigid, every object we use below will be rigid as long as it comes from~$\cat{K}$.
\end{Rem}

\begin{Rem}
\label{rem:non-rigid}%
In a not-necessarily rigid tt-category, an object~$x$ with empty support, $\supp(x)=\varnothing$, is \tnil, \ie $x\potimes{n}=0$ for some $n\ge 1$. See~\cite[Cor.\,2.4]{Balmer05a}. When $x$ is rigid, $x\potimes{n}=0$ forces $x=0$ since $x$ is a summand of $x\potimes{n}\otimes (x^\vee)\potimes{(n-1)}$.
\end{Rem}

\smallbreak

We begin with Theorem~\ref{thm:surj-closed}, which is relatively straightforward. We only need a few standard facts from basic tt-geometry, which do not use rigidity, namely:
\begin{enumerate}[\rm(A)]
\item
\label{it:A}%
Given a $\otimes$-multiplicative class~$S$ of objects in~$\cat{K}$ (\ie $\unit\in S$ and $x,y\in S \then x\otimes y\in S$) and a tt-ideal $\cat{J}\subset\cat{K}$ such that $\cat{J}\cap S=\varnothing$, then there exists a prime $\cP\in\SpcK$ such that $\cat{J}\subseteq\cP$ and $\cP\cap S=\varnothing$. This fact uses that $\cat{K}$ is essentially small and is proven in~\cite[Lemma~2.2]{Balmer05a}.
\smallbreak
\item
\label{it:B}%
A point $\cP\in\SpcK$ is closed if and only $\cP$ is a \emph{minimal} prime for inclusion in~$\cat{K}$ (\ie $\cP'\subseteq\cP\then \cP'=\cP$). See~\cite[Prop.\,2.9]{Balmer05a}.
\smallbreak
\item
\label{it:C}%
Any non-empty closed subset, for instance $\adhpt{\cP}$ for a point~$\cP$, or $\supp(x)$ for a non-trivial object~$x$, contains a closed point. See~\cite[Cor.\,2.12]{Balmer05a}.
\smallbreak
\item
\label{it:D}%
For $F\colon\cat{K}\to \cat{L}$ and $\varphi=\Spc(F)\colon\SpcL\to \SpcK$, and every object~$x\in\cat{K}$, we have $\supp(F(x))=\varphi\inv(\supp(x))$ in~$\SpcL$. See~\cite[Prop.\,3.6]{Balmer05a}.
\end{enumerate}

\begin{proof}[Proof of Theorem~\ref{thm:surj-closed}]
Suppose that $F\colon\cat{K}\to \cat{L}$ is conservative and let $\cP\in\SpcK$ be a closed point, \ie a minimal prime. Consider its complement $S=\cat{K}\oursetminus\cP$. Since $\cP$ is prime, $S$ is $\otimes$-multiplicative in~$\cat{K}$ and does not contain zero. Since $F$ is a conservative tensor functor, the same holds for the class $F(S)$ in~$\cat{L}$. (Recall that for a triangulated functor~$F$, conservativity is equivalent to $F(x)=0\then x=0$, since a morphism is an isomorphism if and only if its cone is zero.) By the general fact~\eqref{it:A} recalled above, for the $\otimes$-multiplicative class $F(S)$ and for the tt-ideal $\cat{J}=0$ in~$\cat{L}$, there exists a prime~$\cQ\in\SpcL$ such that $\cQ\cap F(S)=\varnothing$. This relation implies that $F\inv(\cQ)\subseteq \cP$. By minimality of the closed point~$\cP$, see~\eqref{it:B}, this inclusion $F\inv(\cQ)\subseteq \cP$ forces $\cP=F\inv(\cQ)=\varphi(\cQ)$.

Conversely, suppose that $\varphi\colon\SpcL\to \SpcK$ is surjective on closed points and let $x\in\cat{K}$ be such that $F(x)=0$. We want to show that $x=0$. Suppose \ababs\ that $x\neq 0$. Then we have $\supp(x)\neq\varnothing$. By~\eqref{it:C}, we know that there exists a closed point $\cP\in\supp(x)$, which by assumption belongs to the image of~$\varphi$, say $\cP=\varphi(\cQ)$. But then $\cQ\in\varphi\inv(\supp(x))=\supp(F(x))$ by~\eqref{it:D}. This last statement contradicts $\supp(F(x))=\supp(0)=\varnothing$. So $x=0$ as claimed.
\end{proof}

\begin{Rem}
The proof also gives a statement for $\cat{K}$ not rigid. In that case, the property $\supp(x)=\varnothing$ does not necessarily imply that $x=0$ but that $x$ is \tnil, as an object. See Remark~\ref{rem:non-rigid}. Surjectivity of $\varphi$ onto closed points is therefore equivalent to $F$ detecting \tnilce\ of objects. See Theorem~\ref{thm:surj-closed}\,(\ref{it:a'}').
\end{Rem}

\begin{Rem}
In complete generality, if a closed point $\cP\in\SpcK$ belongs to the image of $\varphi\colon\SpcL\to \SpcK$, say $\cP=\varphi(\cQ)$, then $\cP$ is also the image of a \emph{closed} point~$\cQ'$, which can be chosen in the closure of~$\cQ$. Indeed, there exists a closed point $\cQ'\in\adhpt{\cQ}$ by~\eqref{it:C} and continuity of~$\varphi$ implies $\varphi(\cQ')\in\adhpt{\cP}=\{\cP\}$.
\end{Rem}

\begin{center}*\ *\ *\end{center}

We now turn to the slightly more tricky Theorem~\ref{thm:main}. Let us clarify the following:
%
\begin{Def}
\label{def:tens-nil}%
A morphism $f\colon x\to y$ is called \emph{\tnil}\  if $f\potimes{n}\colon x\potimes{n}\to y\potimes{n}$ is zero for some~$n\ge 1$.
We say that $f\colon x\to y$ is \emph{$\otimes$-nilpotent\ on an object~$z$} in~$\cat{K}$ if there exists $n\ge 1$ such that $f\potimes{n}\otimes z$ is the zero morphism $x\potimes{n}\otimes z\to y\potimes{n}\otimes z$. In particular, $f$ is \emph{\tnil\ on its cone} if there exists $n\ge 1$ such that $f\potimes{n}\otimes \cone(f)=0$.
\end{Def}

The following useful fact was already observed in~\cite[Prop.\,2.12]{Balmer10b}:
\begin{Prop}
\label{prop:nil-tt-ideal}%
Let $f\colon x\to y$ be a morphism in~$\cat{K}$. Then
\[
\SET{z\in\cat{K}}{f\textrm{ is \tnil\ on }z}
\]
forms a tt-ideal, even if $\cat{K}$ is not rigid.
\end{Prop}

\goodbreak

Closure under direct summands and $\otimes$ is clear from the definition. The trick for closure under cones, is that if $f\potimes{n_i}\otimes z_i=0$ for $i=1,2$ and if $z_1\to z_2\to z_3\to \Sigma z_1$ is an exact triangle, then $f\potimes{(n_1+n_2)}\otimes z_3$ will vanish. This is the place where the same statement would fail with `$f$ vanishes on~$z$' (instead of `$f$ \tnil\ on~$z$').

\begin{Prop}
\label{prop:spectra^3}%
Let $\xi\colon w\to \unit$ be a morphism in~$\cat{K}$ (not necessarily rigid) such that $\xi\otimes\cone(\xi)=0$. Then the cone of~$\xi\potimes{n}$ generates the same tt-ideal, for all~$n$:
\[
\ideal{\cone(\xi)}=\SET{z\in\cat{K}}{\xi\textrm{ is \tnil\ on }z}=\ideal{\cone(\xi\potimes{n})}\,.
\]
\end{Prop}

\begin{proof}
The assumption $\xi\otimes \cone(\xi)=0$ implies that the object $\cone(\xi)$ belongs to~$\SET{z\in\cat{K}}{\xi\textrm{ is \tnil\ on }z}$, which is a tt-ideal by Proposition~\ref{prop:nil-tt-ideal}. On the other hand, if the morphism $\xi\potimes{n}\otimes z$ is zero then the exact triangle
\[
\xymatrix@C=2em{
w\potimes{n}\otimes z \ar[rr]^-{\xi\potimes{n}\otimes z=0}
&& z \ar[r]
& \cone(\xi\potimes{n})\otimes z \ar[r]
& \Sigma w\potimes{n}\otimes z
}
\]
implies that $z$ is a summand of~$\cone(\xi\potimes{n})\otimes z$. Hence $z$ belongs to~$\ideal{\cone(\xi\potimes{n})}$. Finally, in the Verdier quotient $\cat{K}/\ideal{\cone(\xi)}$, the morphism $\xi$ is an isomorphism, hence so is~$\xi\potimes{n}$. Therefore $\cone(\xi\potimes{n})\in\ideal{\cone(\xi)}$. In short, we have obtained
\[
\ideal{\cone(\xi)} \subseteq \SET{z\in\cat{K}}{\xi\potimes{n}\otimes z=0\textrm{ for some }n\ge1}\subseteq \cup_{n\ge1}\ideal{\cone(\xi\potimes{n})}
\subseteq\ideal{\cone(\xi)}\,.
\]
%
This proves the claim. Compare~\cite[\S\,2]{Balmer10b}.
\end{proof}

We can now establish the key observation of the paper:
\begin{Cor}
\label{cor:key}%
Let $x\in\cat{K}$ be a rigid object in a (not necessarily rigid) tt-category~$\cat{K}$. Choose $\xi_x$ a `homotopy fiber' of the coevaluation morphism~$\eta_x$ of~\eqref{eq:(co)units}, \ie choose an exact triangle in~$\cat{K}$
\begin{equation}
\label{eq:triangle}
\xymatrix{
w_x \ar[r]^-{\displaystyle\xi_x}
& \unit \ar[r]^-{\displaystyle\eta_x}
& x^\vee\otimes x \ar[r]^-{}
& \Sigma w_x
}
\end{equation}
for a morphism~$\xi_x$. Then the tt-ideal $\ideal{x}$ generated by our object is exactly the subcategory on which $\xi_x$ is $\otimes$-nilpotent:
\begin{equation}
\label{eq:key}%
\ideal{x} = \SET{z\in\cat{K}}{\xi_x\potimes{n}\otimes z=0\textrm{ for some }n\ge1}\,.
\end{equation}
Moreover, for every $n\ge1$ the morphism $\xi_x\potimes{n}$ is \tnil\ on its cone.
\end{Cor}

\begin{proof}
Consider the exact triangle obtained by tensoring~\eqref{eq:triangle} with~$x$:
\[
\xymatrix@C=4em{
x\otimes w_x \ar[r]^-{x\otimes \xi_x}
& x \ar[r]^-{x\otimes \eta_x}
& x\otimes x^\vee\otimes x \ar[r]^-{x\otimes \zeta_x}
& \Sigma x\otimes w_x
}
\]
By the unit-counit relation~\eqref{eq:(co)unit-relation}, the morphism $x\otimes \eta_x$ is a monomorphism. This forces $x\otimes \xi_x=0$. Hence $\xi_x\otimes \cone(\xi_x)\simeq \xi_x\otimes x^\vee\otimes x=0$ and we can apply Proposition~\ref{prop:spectra^3} to $\xi=\xi_x$. It gives us~\eqref{eq:key} since $\ideal{\cone(\xi_x)}=\ideal{x^\vee\otimes x}=\ideal{x}$ by rigidity of~$x$. The `moreover part' also follows from Proposition~\ref{prop:spectra^3} where we proved that $\xi$ is \tnil\ on~$\cone(\xi\potimes{n})$.
\end{proof}

The above result allows us to translate questions about tt-ideals into a \tnilce\ problem. We isolate a surjectivity argument that we shall use twice.
\begin{Lem}
\label{lem:key}%
Under Hypotheses~\ref{hyp}, choose for every $x\in\cat{K}$ an exact triangle as in~\eqref{eq:triangle}. Let $\cP\in\SpcK$ be a prime. Suppose that $\cP$ satisfies the following technical condition:
\begin{equation}
\label{eq:N}%
\textrm{For all }x\in \cP,\textrm{ all }s\in\cat{K}\oursetminus\cP\textrm{ and all }n\ge1,\textrm{ we have }F(\xi_x\potimes{n}\otimes s)\neq0.
\end{equation}
Then $\cP$ belongs to the image of~$\varphi\colon\SpcL\to \SpcK$.
\end{Lem}

\begin{proof}
Consider the complement $S=\cat{K}\oursetminus \cP$. Let $\cat{J}\subseteq\cat{L}$ be the tt-ideal generated by~$F(\cP)$, just viewed as a class of objects in~$\cat{L}$. We claim that $\cat{J}=\ideal{F(\cP)}$ equals
\[
\cat{J}':=\SET{y\in\cat{L}}{\textrm{there exists }x\in\cP\textrm{ such that }y\in\ideal{F(x)}}.
\]
Indeed, since we have $F(\cP)\subseteq\cat{J}'\subseteq\cat{J}$ directly from the definitions, it suffices to show that $\cat{J}'$ is a tt-ideal. It is clearly thick and a $\otimes$-ideal. For closure under cones, if $y_1\to y_2\to y_3\to \Sigma y_1$ is exact in~$\cat{L}$ and $y_i\in\ideal{F(x_i)}$ for $x_i\in\cP$ and $i=1,2$, then $y_3\in\ideal{y_1,y_2}\subseteq\ideal{F(x_1),F(x_2)}=\ideal{F(x_1\oplus x_2)}$ and $x_1\oplus x_2$ still belongs to~$\cP$.

Now, for every object~$x\in\cat{K}$, the tt-functor $F\colon\cat{K}\to \cat{L}$ sends an exact triangle over the unit~$\eta_x$ as in~\eqref{eq:triangle} to an exact triangle in~$\cat{L}$:
\[
\xymatrix@C=4em{
F(w_x) \ar[r]^-{F(\xi_x)}
& \unit \ar[r]^-{\eta_{F(x)}}
& F(x)^\vee\otimes F(x) \ar[r]^-{F(\zeta_x)}
& \Sigma F(w_x)\,.
}
\]
Here we use that $F(\eta_x)=\eta_{F(x)}$ which is another way of saying that $F$ preserves duals. See Remark~\ref{rem:rigid}. Using this last exact triangle in Corollary~\ref{cor:key} for the rigid object~$F(x)$ in the tt-category~$\cat{L}$, we see that
\[
\ideal{F(x)}=\SET{y\in\cat{L}}{F(\xi_x)\potimes{n}\otimes y=0\textrm{ for some }n\ge 1}\,.
\]
Combining this with the description of~$\cat{J}=\ideal{F(\cP)}$ as $\cat{J}'$ above, we obtain
\[
\ideal{F(\cP)}=\SET{y\in\cat{L}}{F(\xi_x)\potimes{n}\otimes y=0\textrm{ for some }n\ge 1\textrm{ and some }x\in\cP}\,.
\]
It follows that if $s\in S=\cat{K}\oursetminus\cP$ then $F(s)$ cannot belong to~$\cat{J}=\ideal{F(\cP)}$. Indeed, if $F(s)\in\ideal{F(\cP)}$ then by the above there exists $x\in\cP$ and $n\ge1$ such that $0=F(\xi_x)\potimes{n}\otimes F(s)\cong F(\xi_x\potimes{n}\otimes s)$ since $F$ is a $\otimes$-functor. This contradicts~\eqref{eq:N}.

In short, we have shown that the $\otimes$-multiplicative class $F(S)=F(\cat{K}\oursetminus \cP)$ does not meet the tt-ideal $\cat{J}=\ideal{F(\cP)}$, in the tt-category~$\cat{L}$. By the existence trick~\eqref{it:A} again, there exists a prime~$\cQ$ satisfying the following two relations: $\cat{J}\subseteq \cQ$ and $F(S)\cap \cQ=\varnothing$. Unpacking the definition of $S=\cat{K}\oursetminus\cP$ and $\cat{J}=\ideal{F(\cP)}$, these two relations mean respectively $\cP\subseteq F\inv(\cQ)$ and $F\inv(\cQ)\subseteq \cP$. Hence $\cP=F\inv(\cQ)=\varphi(\cQ)$ as wanted.
\end{proof}

We are now ready to prove our main result.

\begin{proof}[Proof of Theorem~\ref{thm:main}]
\

(a)$\then$(b): Suppose that $\varphi\colon\SpcL\to \SpcK$ is surjective and let $f\colon x\to y$ be a morphism such that $F(f)=0$ and which is \tnil\ on its cone, say $f\potimes{m}\otimes\cone(f)=0$. It follows from the exact triangle $x\oto{f} y\to \cone(f)\to \Sigma x$ in~$\cat{K}$ and from $F(f)=0$ that $F(\cone(f))\simeq F(y)\oplus \Sigma F(x)$ in~$\cat{L}$. Taking supports, we have $\supp(F(\cone(f)))=\supp(F(x))\cup \supp(F(y))$. By~\eqref{it:D}, this translates into
\[
\varphi\inv(\supp(\cone(f)))=\varphi\inv(\supp(x))\cup \varphi\inv(\supp(y))=\varphi\inv(\supp(x)\cup \supp(y))\,.
\]
Since $\varphi$ is surjective, this implies $\supp(\cone(f))=\supp(x)\cup \supp(y)$. Therefore $x,y\in\ideal{\cone(f)}$. But we assumed that $f$ is \tnil\ on~$\cone(f)$ and it follows from Proposition~\ref{prop:nil-tt-ideal} that $f$ is also \tnil\ on~$x$ and on~$y$. This means that there exists $n\ge 1$ such that $f\potimes{n}\otimes x=0\colon x\potimes{(n+1)}\to y\potimes{n}\otimes x$. But then $f\potimes{(n+1)}$ decomposes as
\[
\xymatrix@C=5em{
x\potimes{n+1} \ar[r]_-{f\potimes{n}\otimes x=0} \ar@/^1em/[rr]^-{f\potimes{(n+1)}}
& y\potimes{n}\otimes x \ar[r]_-{y\potimes{n}\otimes f}
& y\potimes{(n+1)}
}
\]
and is therefore also zero, that is, $f\potimes{(n+1)}=0$ as wanted.

(b)$\then$(a):
Suppose that $F\colon\cat{K}\to \cat{L}$ detects \tnilce\ of those morphisms which are already zero on their cone. Let $\cP\in\SpcK$ be a prime and let us show that property~\eqref{eq:N} in Lemma~\ref{lem:key} is satisfied. Let $g=\xi_x\potimes{n}\otimes s$ be the morphism in~\eqref{eq:N} for some objects $x\in\cP$ and $s\in\cat{K}\oursetminus\cP$ and for $n\ge1$. Suppose \ababs\ that $F(g)=0$. The cone of~$g=\xi_x\potimes{n}\otimes s$ is simply $\cone(\xi_x\potimes{n})\otimes s$. By Corollary~\ref{cor:key}, $\xi_x\potimes{n}$ is \tnil\ on its cone. Hence $g$ is \tnil\ on its cone as well. We can therefore apply our assumption~\eqref{it:nil-cone} to~$g$ and deduce from the (absurd) assumption $F(g)=0$ that $g=\xi_x\potimes{n}\otimes s$ is \tnil. In other words, $\xi_x$ is $\otimes$-nilpotent on~$s\potimes{m}$ for some $m\ge1$. By Corollary~\ref{cor:key} again, this implies that $s\potimes{m}$ belongs to~$\ideal{x} \subseteq\cP$, and therefore $s\in\cP$ since $\cP$ is prime, a contradiction with the choice of~$s$ in~$S=\cat{K}\oursetminus\cP$. In short, we have verified property~\eqref{eq:N} of Lemma~\ref{lem:key} for the prime~$\cP$, which tells us that $\cP$ belongs to the image of~$\varphi$ as claimed.
\end{proof}

\begin{center}*\ *\ *\end{center}

Let us now prove Theorems~\ref{thm:FU} and~\ref{thm:FU+}. We therefore assume the existence of an adjoint~$U\colon \cat{L}\to \cat{K}$ to our tensor-triangulated functor~$F$:
\begin{equation}
\label{eq:FU}%
\vcenter{\xymatrix@R=2em{
\cat{K} \ar@<-.5em>[d]_-{F} \ar@{}[d]|-{\adj}
\\
\cat{L} \ar@<-.5em>[u]_-{U}
}}
\end{equation}
By general theory, $U$ must satisfy a projection formula
\begin{equation}
\label{eq:proj-form}%
U(F(x)\otimes z)\cong x\otimes U(z)
\end{equation}
for all~$x\in\cat{K}$ and $z\in\cat{L}$. The latter is an easy consequence of rigidity of~$x$ and the adjunctions~\eqref{eq:adj-dual} and~\eqref{eq:FU}. See for instance~\cite[Prop.\,3.2]{FauskHuMay03}.

\begin{proof}[Proof of Theorem~\ref{thm:FU+}]
Let $\cP\in\SpcK$. We need to show that $\cP\in\im(\varphi)$ if and only if~$\cP\in\supp(U(\unit))$. The latter means $U(\unit)\notin\cP$.

Suppose first that $\cP=\varphi(\cQ)$ for some $\cQ\in\SpcL$. Then $\cP=F\inv(\cQ)$. To show $U(\unit)\notin\cP$ it therefore suffices to show that $FU(\unit)\notin\cQ$. This is easy since, by the unit-counit relation for~\eqref{eq:FU}, the object $FU(\unit_{\cat{L}})\cong FUF(\unit_{\cat{K}})$ admits $F(\unit_{\cat{K}})\cong\unit_{\cat{L}}$ as a direct summand and $\unit$ cannot belong to any prime.

The reverse inclusion is the interesting one. So, let $\cP\in\supp(U(\unit))$, meaning $U(\unit)\notin\cP$. Let us show that $\cP$ satisfies condition~\eqref{eq:N} of Lemma~\ref{lem:key}. Take objects $x\in\cP$ and $s\in\cat{K}\oursetminus\cP$, and suppose \ababs\ that $F(g)=0$ where $g=\xi_x\potimes{n}\otimes s$ for some $n\ge 1$ as before. By the projection formula~\eqref{eq:proj-form} for~$z=\unit$, the property $UF(g)=U(0)=0$ implies $g\otimes U(\unit)=0$. Consequently we have an exact triangle
\[
\xymatrix@C=2em{
w_x\potimes{n}\otimes s\otimes U(\unit) \ar[rr]^-{g\otimes U(\unit)=0}
&& s\otimes U(\unit) \ar[r]
& \cone(g)\otimes U(\unit) \ar[r]
& \Sigma w_x\potimes{n}\otimes s\otimes U(\unit)
}
\]
in~$\cat{K}$. This proves that $s\otimes U(\unit)$ is a direct summand of~$\cone(g)\otimes U(\unit)\in \ideal{\cone(g)} \subseteq\ideal{\cone(\xi_x\potimes{n})}$. By Proposition~\ref{prop:spectra^3}, the latter is contained in~$\ideal{x} \subseteq\cP$. In short, we have $s\otimes U(\unit)\in\cP$. Since $\cP$ is prime this forces $s\in\cP$ or $U(\unit)\in\cP$, which are both absurd. So we have proven~\eqref{eq:N} for~$\cP$ and we conclude by Lemma~\ref{lem:key} again.
\end{proof}

\begin{proof}[Proof of Theorem~\ref{thm:FU}]
In view of Theorem~\ref{thm:FU+} it suffices to prove that $F\colon\cat{K}\to \cat{L}$ detects \tnilce\ if and only if $\supp(U(\unit))=\SpcK$, which means $\ideal{U(\unit)}=\cat{K}$. This is a standard argument, as in~\cite[Prop.\,3.15]{Balmer16a} for instance. Let us outline it for completeness. The point is that $A:=U(\unit)$ is a ring-object (for $U$ is lax-monoidal). Let $J \oto{\xi} \unit \oto{u} A \to \Sigma J$ be an exact triangle over the unit $u\colon\unit \to A$ (the unit of the $F\adj U$ adjunction at~$\unit$). We have $A\otimes \xi=0$ (since $A\otimes u$ is a split monomorphism, retracted by multiplication~$A\otimes A\to A$). A morphism $f\colon x\to y$ satisfies $F(f)=0$ if and only if the composite $x\oto{f}y\oto{u\otimes y}A\otimes y$ is zero (by adjunction and the projection formula: $A\otimes-\simeq UF(-)$); this is in turn equivalent to the morphism $f\colon x\to y$ factoring via~$\xi\otimes y\colon J\otimes y\to y$ (by the exact triangle $J\otimes y\otoo{\xi\otimes y}y\otoo{u\otimes y}A\otimes y\otoo{}\Sigma J\otimes y$). So we are down to proving that $\xi\colon J\to \unit$ is \tnil\ if and only if~$\ideal{A}=\cat{K}$. This is now immediate from Proposition~\ref{prop:spectra^3}, which says that $\ideal{A}=\SET{z\in\cat{K}}{\xi\textrm{ is \tnil\ on }z}$. Indeed, $\unit\in\ideal{A}$ if and only if~$\xi$ is \tnil\ on~$\unit$.
\end{proof}


\end{document}